\documentclass[11pt]{amsart}

\usepackage{amsmath}
\usepackage{amssymb}
\usepackage{amsthm}
\usepackage{graphicx}
\usepackage{tikz}
\usepackage{xcolor}

\allowdisplaybreaks

\usepackage[margin=1.2in]{geometry}
\usepackage{bookmark}

\newtheorem{theorem}{Theorem}
\newtheorem*{theorem*}{Theorem}

\newtheorem*{conjecture*}{Conjecture}
\newtheorem{corollary}[theorem]{Corollary}
\newtheorem{lemma}[theorem]{Lemma}

\newtheorem{proposition}[theorem]{Proposition}
\newtheorem{definition}[theorem]{Definition}
\newtheorem{claim}[theorem]{Claim}
\newtheorem*{claim*}{Claim}
\newtheorem{fact}[theorem]{Fact}

\newcommand\numberthis{\addtocounter{equation}{1}\tag{\theequation}}

\newcommand{\E}{\mathbb{E}}
\newcommand{\Prob}{\mathbb{P}}
\begin{document}

\title{Enumerating independent sets in Abelian Cayley graphs}

\author{
  Aditya Potukuchi}\thanks{ 
    	Department of Mathematics, Statistics, and Computer Science, University of Illinois at Chicago, 851 S Morgan St, Chicago, IL, USA.
    	Email: \texttt{adityap}@\texttt{uic.edu}. Research supported in part by NSF grant CCF-1934915.
    	}
\author{
     	Liana Yepremyan}\thanks{ 
    	Department of Mathematics, London School of Economics, London WC2A 2AE, UK.
    	Email: \texttt{l.yepremyan}@\texttt{lse.ac.uk}. Research supported by Marie Sklodowska Curie Global Fellowship, H2020-MSCA-IF-2018:846304.
	}

\date{\today}

\maketitle

\begin{abstract}
We show that any connected Cayley graph $\Gamma$ on an Abelian group of order $2n$ and degree $\tilde{\Omega}(\log n)$ has at most $2^{n+1}(1 + o(1))$ independent sets. This bound is tight up to to the $o(1)$ term when $\Gamma$ is bipartite. Our proof is based on Sapozhenko's graph container method and uses the Pl\"{u}nnecke-Rusza-Petridis inequality from additive combinatorics.
\end{abstract}


\section{Introduction}

An \emph{independent set} in a graph is a set of vertices with no two having an edge between them. For a graph $G$, let $i(G)$ denote the number of independent sets in a graph $G$. The study of $i(G)$ in a $d$-regular graph on a given number of vertices goes back to Granville, who was interested in this quantity because of connections to combinatorial group theory.  In 1988 at a Number Theory Conference in Banff, he suggested that if $G$ is a $d$-regular graph on $2n$ vertices then $i(G)\leq 2^{(1+o(1))n}$, where the $o(1)$ term goes to zero as $d$ goes to infinity. Note that this is tight up to the $o(1)$ term since a bipartite $d$-regular graph has at least  $2^{n +1}-1$ independent sets, just by counting all subsets of both sides in the bipartition. Alon~\cite{alon91} settled this conjecture and proved that $i(G) \leq 2^{\left(1+O(d^{-0.1})\right)n}$. He also suggested that the right bound is $(2^{d+1}-1)^{n/d}$, achieved by a disjoint union of $n/d$ complete bipartite graphs $K_{d,d}$, whenever $d$ divides $n$. This was later also conjectured by Kahn~\cite{kahn2001} and proved for bipartite graphs. The full conjecture was proved by Zhao~\cite{ZHAO10} who showed that the bound for general $d$-regular graphs follows from the bipartite version.

For irregular graphs, Kahn conjectured~\cite{kahn2001} that a similar bound must hold, more precisely, 
 $i(G)\leq \prod_{uv\in E(G)}i(K_{d_G(u),d_G(v)})^{1/d_G(u)d_G(v)},$ and equality holds when $G$ is a union of vertex disjoint complete bipartite graphs with appropriate sizes. It was first proved to be true for all graphs of maximum degree at most $5$ by Galvin and Zhao~\cite{GZ2011} with computer assistance. The full conjecture was recently proved by  Sah, Sawhney, Stoner and Zhao~\cite{SSSZ19} using a H\"older-type inequality. 
  

Back to $d$-regular graphs, the results of Kahn and Zhao show that for any $d$-regular graph on $2n$ vertices, $i(G) \leq 2^{n + \frac{n}{d} + o(1)}$, with the extremal example being a union of complete bipartite graphs. So, a natural question is weather $i(G)$ is much smaller if we require the extremal graph to  have higher edge-connectivity. More specifically, can the $n/d$ term be significantly reduced? The answer is no, and in Appendix~\ref{sec:conn}, we describe the construction of a $d$-regular graph on $2n$ vertices that is at least $(d-1)$ edge-connected and has at least $2^{n + \Omega(n/d)}$ independent sets.

However, if we require the extremal graph to have some stronger connectivity properties, such as being an expander, then more is known. For example,  for  the $d$-dimensional discrete hypercube, $Q_d$, that is the graph on vertex set $\{0,1\}^d$ where two vertices are adjacent if they differ in exactly one coordinate, Korshunov and Sapozhenko~\cite{KS83} proved that  $i(Q_d)=2\sqrt{e} (1+o(1)) \cdot 2^{2^{d-1}}$, when $d\rightarrow \infty$.  Sapozhenko~\cite{SAPOZHENKO87} using the container method gave a simplified proof of this result (see ~\cite{GALVIN19} for a beautiful exposition of this method). The ideas introduced in this method have proved to be extremely useful, finding a number of applications in combinatorics. upper bound on phase transition on the hardcore model on $\mathbb{Z}^d$~\cite{GK04}, lower bounds for mixing for Glauber dynamics for hardcore model in bipartite regular graphs~\cite{GT06}, enumerating uniform intersecting set systems~\cite{BGLW21}, enumerating $q$-colorings of the discrete torus~\cite{GALVIN03},~\cite{KP20a},~\cite{JK20}, phase coexistence of the $3$-coloring model in $\mathbb{Z}^d$~\cite{GKRS15}, more detailed descriptions of independent sets in the hypercube~\cite{BGL21},~\cite{GALVIN10},~\cite{JP20},~\cite{JPP21a},~\cite{KP19},~\cite{PARK21}, and faster algorithms for approximately counting independent sets in bipartite graphs~\cite{JPP21b}.

\section{Our results}

The motivation of this paper is to consider the question of determining $i(G)$ for families of graphs with some underlying structure. A natural example are Cayley graphs. Let $\mathcal{F}$ denote a finite Abelian group. An \emph{Abelian Cayley graph} for $\mathcal{F}$ with generator set $D\subseteq \mathcal{F}$, is a graph whose vertices are given by the elements of $\mathcal{F}$, and (directed) edges $\{(u,u+x)~|~ x \in D\}$. If $D = -D$, then we may assume the graph is undirected. Our main result is the following.

\begin{theorem}\label{thm:main} 
Let $\Gamma$ be a connected Abelian Cayley graph on $2n$ vertices and degree \linebreak $\Omega(\log n \cdot (\log\log{n})^{11})$. Then,
\[
i(\Gamma) \leq 2^{n+1} \cdot \left(1+ o(1)\right).
\]
\end{theorem}

This is asymptotically tight whenever $\Gamma$ is bipartite. In this case, the theorem says that most independent sets come from subsets of either part.

Apart from exhibiting this property in Cayley graphs, extending the aforementioned techniques to graphs where the guarantees in typical uses of container method are unavailable seems to be of independent interest.

 In Theorem~\ref{thm:main} make no attempt to optimize the lower bound on $d$ for the conclusion of the Theorem to hold. However, we are unable to reduce this to $\Omega(\log n)$, which we believe is the truth. In fact, we conjecture:

\begin{conjecture*}
Fix $\epsilon>0$ and let $\Gamma$ be a connected Abelian Cayley graph on $2n$ vertices and degree $(2+ \epsilon)\log n$. Then,
\[
i(\Gamma) \leq 2^{n+1} \cdot (1+o(1)).
\]
\end{conjecture*}

This conjecture, if true, would be optimal. The natural guess for the tight case is the example described in Appendix~\ref{sec:appendix}. We use the looser bound of $\tilde{\Omega}(\log n)$ in a couple of places in the proof, but we believe that the main bottleneck is in  Lemma~\ref{lem:bdry}. 

Let us set up some basic notation for the proof. In what follows, the notation is focused on bipartite graphs. The reason is that in the proof of Theorem~\ref{thm:main}, that the main part is the case when $\Gamma$ is bipartite. The non-bipartite case is handled using a theorem of Zhao~\cite{ZHAO10}.

Let us now restrict ourselves to the case when $\Gamma$ is bipartite, and use $(X,Y)$ to denote the bipartition, where $|X| = |Y| = n$. Since the graph is connected and $|D|$-regular, for any $S \subseteq X$, we have $|N(S)| \geq |S|$ with equality holding if and only if $S = X$. One can verify that the graph is connected if and only if the set $D$ is a \emph{generating set} of the group, i.e., every element in $\mathcal{F}$ can be written as a sum of elements from $D$. Throughout the paper, we use sumset notation: For  $A,B \subseteq \mathcal{F}$, we use $A+B$ to denote the set $\{a+b~|~ a \in A,~b \in B\}$, and $2A = A+A$. Thus the set of neighbors  of a set $A$, $N_{\Gamma}(A)$ is just the set $A+D$.

Let $\Gamma^2$ denote the square graph of $\Gamma$, i.e., $V(\Gamma^2) = V(\Gamma)$ and $u \sim_{\Gamma^2}v$ if and only if $N_{\Gamma}(u) \cap N_{\Gamma}(v) \neq \emptyset$. For a subset $A \subset X$, we use $G$ to denote $N(A)$. We say that $A$ is $2$-linked if $A$ is connected in $\Gamma^2$. Let us define $[A] := \{u \in X~|~N(u) \subseteq G\}$ to be the \emph{closure} of $A$. If $|[A]| \leq n/2$, we call $A$ \emph{small}. Let $\mathcal{G}(a,g)$ denote the number of small $2$-linked sets $A$ such that $|[A]| = a$, and $|G| = g$. Let us define $t = g - a$. The main lemma in the proof is the following:

\begin{lemma}
\label{lem:main}
If $\Gamma$ is an bipartite connected Cayley graph with bipartition $(X,Y)$ with $|X| = |Y| = n$, and generator set $D$, such that  
\begin{enumerate}
\item $\frac{|D|}{\log^{8}|D|} = \Omega(\log n)$,
\item $|D| \leq n^{1/3}$, and 
\item $|2D| \geq |D| \log^3 |D|$
\end{enumerate}
Then we have for every $a,g$
\[
|\mathcal{G}(a,g)| \leq 2^{g - \Omega \left({t}\right)}.
\]
\end{lemma}

\subsection*{Remark}

For comparison, the graph container lemma of Sapozhenko, improved by Kahn and Park~\cite{KP19} says the following:

\begin{lemma}
If $\Gamma$ is a $d$-regular bipartite graph with $d \gg \log n$ and bipartition $(X,Y)$ such that
\begin{enumerate}
\item Every two vertices have at most $O(1)$ common neighbors 
\item For every small set $A \subset X$, we have that $t \geq \Omega \left(g\frac{\log^2d}{d^2}\right)$,
\end{enumerate}
Then we have for every $a,g$
\[
|\mathcal{G}(a,g)| \leq 2^{g - \Omega \left({t}\right)}.
\]
\end{lemma}

Condition $2$ imposes certain expansion conditions on the graph, which is not true in general for Cayley graphs. We overcome this using tools from additive combinatorics.

\subsection{Organization}
In section~\ref{sec:preliminaries}, we state some results from additive combinatorics that are useful. The end of this section contains the proof of Theorem~\ref{thm:main} using Lemma~\ref{lem:main}. Section~\ref{sec:mainlem} is dedicated to the proof of Lemma~\ref{lem:main}. The proof of Lemma~\ref{lem:main}, using Lemma~\ref{lem:bdry}, Lemma~\ref{lem:varphi}, and Lemma~\ref{lem:psi} is given in Subsection~\ref{subsec:recon}. Some preliminary lemmas are proved in Subsection~\ref{subsec:preliminaries}, after which, Subsections~\ref{subsec:bdry},~\ref{subsec:varphi}, and~\ref{subsec:psi} are dedicated to the proofs of Lemmas~\ref{lem:bdry},~\ref{lem:varphi}, and~\ref{lem:psi} respectively.

\section{Preliminaries} 
\label{sec:preliminaries}

Here, we will state some results from additive combinatorics that will be useful to us. The first is the Pl\"{u}nnecke-Rusza-Petridis inequality ~\cite{PLUNNECKE70},~\cite{RUZSA89},~\cite{RUZSA90},~\cite{PETRIDIS14}:

\begin{theorem}[Pl\"{u}nnecke-Ruzsa-Petridis Inequality]
\label{thm:PR}
Let $M,D \subset \mathcal{F}$ such that $|M+D| = \alpha |M|$. Then for any nonnegative integer $j$, there is a subset $M' \subseteq M$ such that $|M'+jD| \leq \alpha^j |M'|$.
\end{theorem}

We will also need a theorem by Olson~\cite{OLSON84} which is a Cauchy-Davenport type theorem for general Abelian groups. 

\begin{theorem}[\cite{OLSON84}]
\label{thm:Olson}
Let $M,N \subseteq \mathcal{F}$ such that $0\in N$. Then either $M + 2N = M+N$ or $|M+N| \geq |M| + |N|/2$.
\end{theorem}
We can easily derive the following from Theorem~\ref{thm:Olson} by  applying it for $N'=N - a$, for some element $a\in N$ such that $0\in N'$.

\begin{corollary}
\label{thm:CDGEN}
Let $M,N \subseteq \mathcal{F}$. Then either $|M + 2N| = |M+N|$ or $|M+N| \geq |M| + |N|/2$.
\end{corollary}

While the Pl\"{u}nnecke-Ruzsa-Petridis inequality as stated, gives no guarantee on the size of the set $M'$ (in the theorem statement), one may obtain such a theorem through repeated applications of a general version of Theorem~\ref{thm:PR}:

\begin{theorem}[\cite{AR09}, part $\text{II}$, Theorem $1.7.3$]
\label{thm:PRGEN}
Let $M,D,N \subset \mathcal{F}$ such that $|M| = m$, and let $1 \leq j<h$ be positive integers, with $\gamma:=h/j$. Let $|(M + jD) \setminus (N + (j-1)D)| = s$, and $\ell < m$ be a positive integer. There is a subset $M' \subseteq M$ such that $|M'| > \ell$, and 
\begin{align*}
|(M' + hD) \setminus (N+ (h-1)D)| & \leq \frac{s^{\gamma}}{\gamma}\left(\frac{1}{(m-\ell)^{\gamma - 1}} - \frac{1}{m^{\gamma - 1}}\right) +  \left(\frac{s}{m - \ell}\right)^{\gamma} \left(|M'| - \ell \right).
\end{align*}
\end{theorem}

In fact, eventually, we will want a set $M' \subseteq M$ such that each $M' + iD$ is small (see Lemma~\ref{lem:PRUSE2}), which may be obtained by repeated application of Theorem~\ref{thm:PRGEN}.

As mentioned before, we need the following theorem of Zhao~\cite{ZHAO10} which allows  to  prove our main result just for bipartite graphs. Here, $\Gamma \times K_2$ is a bipartite graph with vertex set $V(\Gamma) \times \{1,2\}$ and (undirected) edge set $\{\{(u,1)(v,2)\},\{(v,1),(u,2)\}~|~\{u,v\} \in E(\Gamma)\}$.

\begin{theorem}[\cite{ZHAO10}]
\label{thm:ZHAO10}
For any graph $\Gamma$, we have $i(\Gamma \times K_2) \geq i(\Gamma)^2$.
\end{theorem}

We will use the following theorem, originally due to Lov\'{a}sz~\cite{LOVASZ75} and Stein~\cite{Stein74}.

\begin{theorem}[\cite{LOVASZ75}, \cite{Stein74}]
\label{thm:cover}
Let $G$ be a bipartite graph on vertex sets $A$ and $B$ where the degree of each vertex in $A$ is at least $a$ and the degree of each vertex in $B$ is at most $b$. Then there is subset $B' \subset B$ of size at most $\frac{|B|}{a}(1 + \ln b)$ such that $A \subseteq N(B')$.
\end{theorem}

We will also use the following (see for e.g.~\cite{KNUTH98}, p.396, Ex.11).

\begin{proposition}
\label{prop:trees}
The number of rooted trees with maximum degree $d$ and $n$ internal vertices is at most 
\[
\frac{\binom{d n}{n}}{(d-1)n + 1} \leq (ed)^n.
\]
\end{proposition}

We use $\log(\cdot)$ to denote $\log_2(\cdot)$. Finally, throughout the proof, we assume that $n$ (and therefore $d$) is large enough.

\subsection{Proof of Theorem~\ref{thm:main} from Lemma~\ref{lem:main}}

First, a sketch of the proof: Consider a bipartite Cayley graph $\Gamma$ with bipartition $(X,Y)$ with $|X| = |Y| =n$. Every independent set of $\Gamma$ is a subset $A \sqcup B$ such that $A \subseteq X$ and $B \subseteq Y \setminus N(A)$. Moreover, observe that the independence number of $\Gamma$ is $n$ and so one of $A$ or $B$ must have size at most $n/2$. Thus, the total number of independent sets is at most

\begin{equation}
\label{eqn:summ}
\sum_{\substack{A \subset X\\|A| \leq n/2}}2^{n - |N(A)|} + \sum_{\substack{B \subset Y\\|B| \leq n/2}}2^{n - |N(B)|} = 2 \cdot \sum_{\substack{A \subset X\\|A| \leq n/2}}2^{n - |N(A)|}
\end{equation}

where the equality is due to symmetry. The goal is to show that

\[
\sum_{\substack{A \subset X\\\emptyset \neq |A| \leq n/2}}2^{n - |N(A)|} = o(2^{n+1}).
\]

So, the main point behind Lemma~\ref{lem:main} is a way of quantifying the fact that there are not too many sets $A$ for which $2^{-|N(A)|}$ is relatively large.

Henceforth, let $\Gamma$ be a bipartite Cayley graph over an Abelian group $\mathcal{F}$ of order $2n$ and set of generators $D = -D$. We impose a couple of constraints on $D$, namely
\begin{enumerate}
\item $|D| \leq n^{1/3}$, and
\item $|2D| \geq |D| \log^3 |D|$.
\end{enumerate}

We first show that these can be assumed w.l.o.g., when $d = \tilde{\Omega}(\log n)$.

\begin{proposition}
\label{prop:thinning}
Let $D \subseteq \mathcal{F}$ such that $|D| \geq 10 \log{n}(\log{\log{n}})^k$ for some fixed $k > 0$. For any $0< \alpha \leq (\log{\log{n}})^k$, if $|2D| \leq \alpha |D| $ then there is a $D' \subset D$ such that
\begin{enumerate}
\item $D' = -D'$ 
\item $D'$ is a generating set 
\item $|D'| = \Theta\left(\frac{|D|}{\alpha}\right)$
\item $|D' + D'| \geq \alpha |D'| $.
\end{enumerate}
\end{proposition}

\begin{proof}
Set $p := \frac{1}{15 \alpha}$. Choose $P$ to be a $p$-random subset of $D$, and $S \subseteq D$ be a minimal generating set of $\mathcal{F}$.  Note that $|S|\leq \log{n}$. Set

\[
D' = P \cup -P \cup S \cup -S
\]

Property $(1)$ and $(2)$ easily follow from the  definition of $D'$. To see that property $(3)$ holds, we use the Chernoff bound (Theorem $1.1$ in~\cite{DP09}): A binomially distributed variable $X\sim \operatorname{Bin}(n, p)$  for all  $0<a\leq 3/2$ we have
\[
\Prob{[|X-\mathbf{E}[X]|\geq a \mathbf{E}[X]]}\leq 2e^{-\frac{a^2}{3}\mathbf{E}[X]}.
\]
Indeed, $\E[|P|]=p|D|=\frac{|D|}{15\alpha}$. By Chernoff's bounds and using that $|D|=\tilde{\Omega}(\log{n})$, we have that with high probability,  $\frac{|D|}{20\alpha} \leq |P|\leq \frac{|D|}{10 \alpha}$.

Thus property (3) follows from the fact that $|S| \leq \log{n} \leq  \frac{|D|}{10 \alpha}$. So we have, with high probability

\begin{equation}
\label{eqn:concentration}
\frac{|D|}{20 \alpha } \leq |D'| \leq \frac{2|D|}{5\alpha }
\end{equation}.

For property $(4)$: For every $u \in D+D$, define 
\[
R_u := \left\{\{x,y\} \subset D~|~x+y = u\right\},
\]

i.e., the set of \emph{representations} of $u$ in $D+D$. Denote $r_u = |R_u|$. Let 
\[
D_{\ell} = \left\{u \in D+D~|~r_u \geq \frac{|D|}{2\alpha}\right\}.
\]

Using an averaging argument, we have that $|D_{\ell}| \geq |D|/2$. For $u \in D' + D'$, define 
\[
R_u' := \left\{\{x,y\} \subset P~|~x+y = u\right\}.
\]
and $r_u' = |R_u'|$. The main observation is that the elements of $R_u$ are pairwise disjoint. Therefore, for each $u \in D_{\ell}$, we have $\mathbb{P}(u \not \in P+P) = \mathbb{P}(r'_u = 0) \leq (1-p^2)^{r_u} \ll \frac{1}{|D|^2}$. By the Union Bound,  and using the fact that $|D_{\ell}| \leq |2D| \leq |D|^2$, we have that with high probability, every $u \in D_{\ell}$ satisfies $r'_u=0$, therefore the following series of inequalities hold.

\[|D'+D'| \geq |P+P| \geq  |D_{\ell}| \geq |D|/2 \geq \alpha|D'|\]
where the last inequality follows from~(\ref{eqn:concentration}).
\end{proof}

So if $|D| > n^{1/3}$, set $D'' \subset D$ to be an arbitrary subset such that that $D'' = -D''$ and $\frac{n^{1/3}}{2} \leq |D'| \leq n^{1/3}$. Otherwise set $D'' = D$. Now if $|2D''| \leq |D''|\log^3 |D''|$, let $D' \subseteq D''$ be the subset guaranteed by Proposition~\ref{prop:thinning}. Otherwise, set $D' = D''$. Now consider the Cayley graph $\Gamma'$ on $\mathcal{F}$ with the generator set $D' \subseteq D$. Since $\Gamma'$ is a subgraph of $\Gamma$, we have that $2^{n+1} \leq i(\Gamma') \leq i(\Gamma)$. The first inequality is because $X$ and $Y$ are both independent sets of $\Gamma$, each of size $n$. Henceforth, at the cost of a factor of $\frac{1}{\log^{O(1)} |D|}$, we shall assume that the generating set $D$ satisfies $|D| \leq n^{1/3}$ and $|2D| \geq |D| \log^3|D|$. Theorem~\ref{thm:main}, therefore follows from

\begin{theorem}\label{thm:main2}
Let $\Gamma$ be a connected undirected Cayley graph on $2n$ vertices and generating set $D$ such that 
\begin{enumerate}
\item $\frac{|D|}{\log^{8}|D|} = \Omega(\log n)$,
\item $|D| \leq n^{1/3}$, and
\item $|2D| \geq |D| \log^3 |D|$.
\end{enumerate}
Then,
\[
i(\Gamma) \leq 2^{n+1} \cdot (1+o(1)).
\]
\end{theorem}

\begin{proof}[Proof of Theorem~\ref{thm:main2}]
We first prove the theorem for the case when $\Gamma$ be bipartite with bipartition $(X,Y)$. Recall that we say a subset $A \subset X$ or $A \subset Y$ is \emph{small} if $|[A]| \leq n/2$. Let $I \in \mathcal{I}(\Gamma)$ be any independent set. Since $|I| \leq n$, we must have that either $I \cap X$ or $I \cap Y$ is small. Thus we have
\begin{align*}
i(\Gamma) & \leq 2\sum_{A \subseteq X,~\text{small}}2^{n - |N(A)|} \\
& = 2^{n+1}\sum_{A \subseteq X,~\text{small}}2^{-|N(A)|} \\
& \leq 2^{n+1} \sum_k \sum_{\substack{A_1,\ldots,A_k  \subseteq X \\ \text{small, }2-\text{linked}}}2^{- \sum_{i \leq k} |N(A_i)|} \\
& \leq 2^{n+1} \sum_k \frac{1}{k!}\left(\sum_{\substack{A~\text{small, } \\ 2-\text{linked}}}2^{- |N(A)|} \right)^k \\
& \leq 2^{n+1} \exp\left(\sum_{\substack{A~\text{small, } \\ 2-\text{linked}}}2^{- |N(A)|}\right) \\
& = 2^{n+1} \exp\left(\sum_{\substack{1 \leq a \leq n/2\\ d/2 \leq t \leq n}}\sum_{\substack{A~2-\text{linked}\\|[A]| = a, |N(A)| = g}}2^{- g}\right) \\
& = 2^{n+1} \exp\left(\sum_{\substack{1 \leq a \leq n/2\\ d/2 \leq t \leq n}}|\mathcal{G}(a,g)|2^{- g}\right) \\
& \leq 2^{n+1}\exp\left(n \cdot \sum_{d/2 \leq t} 2^{-\Omega(t)}\right) \\
& \leq 2^{n+1}\exp\left(n \cdot 2^{-\Omega(d)}\right) \\
& = 2^{n+1}(1 + o(1)).
\end{align*}

Here, the second last inequality is by Lemma~\ref{lem:main}. The last inequality is because Theorem~\ref{thm:CDGEN} gives that for any set $A$ of size at most $n/2$, $|A+D| \geq |A| + \frac{t}{2}$. Finally, the last (asymptotic) equality follows because $d = \omega(\log n)$.

To handle the case when $\Gamma$ non-bipartite, let $\Gamma' = \Gamma \times K_2$, and observe that $\Gamma'$ is a Cayley graph on $4n$ vertices on the group $\mathcal{F} \times \mathbb{Z}_2$ with generator set $D \times \{1\}$. Moreover, the fact that $\Gamma$ is non-bipartite implies that $\Gamma'$ is also connected (see, for example, Theorem $3.4$ in~\cite{BHM80}). So by Theorem~\ref{thm:ZHAO10} and the preceding proof, we have
\[
i(\Gamma)^2 \leq i(\Gamma') \leq 2^{2n + 1}(1 + o(1))
\]
which gives that $i(\Gamma) \leq 2^{n + 1/2}(1 + o(1))$.
\end{proof}

\section{Proof of Lemma~\ref{lem:main}}
\label{sec:mainlem}
Recall that we are given an undirected bipartite Cayley graph with bipartition $X \cup Y$ with $|X|=|Y|=n$ and generator set $D$. We have the following three conditions on $D$:

\begin{enumerate}
\item $|D| \leq n^{1/3}$,
\item $\log n = O\left(\frac{|D|}{\log^{8}|D|}\right) $, and
\item $|2D| \geq |D| \log^{3} |D|$.
\end{enumerate}

Also recall for a subset $A \subset X$, we say that $A$ is small if $|[A]| \leq n/2$. We use $G$ to denote $N(A)$, with $|[A]| = a$, $|G| = g$, and $t = g - a$. Let us abbreviate $d := |D|$ and $d_2 := |2D|$. The proof of Lemma~\ref{lem:main} has three components, and we describe them here.

Define the \emph{boundary} of $G$ as
\[
G' := \{v \in G~|~ N(v) \cap [A]^c \neq \emptyset\},
\]

which are the vertices in $G$ that are connected to vertices outside $[A]$. The first component to the proof of Lemma~\ref{lem:main} is the following:

\begin{lemma}
\label{lem:bdry} 
There is a family $\mathcal{C}_1 \subset 2^{Y}$ that satisfies the following three properties 
\begin{enumerate}
\item $|C| = O\left(\frac{t d_2}{\log^3 d}\right)$ for every $C \in \mathcal{C}_1$
\item $|\mathcal{C}_1| \leq 2^{O\left(\frac{t}{\log d}\right)}$ 
\item For every $2$-linked $A\subset X$, there is a $C \in \mathcal{C}_1$ such that $G' \subseteq C$.
\end{enumerate}
\end{lemma}

 Lemma~\ref{lem:bdry} offers a starting point from which ideas from the aforementioned container method may be used effectively. The first of these ideas is the notion of ``$\varphi$ -approximation'':

Let $\varphi = d - \frac{\sqrt{d}}{\log d}$, and define for every $\alpha >0$,  $G_{\alpha} := \{u \in G~|~ d_{[A]}(u) \geq \alpha\}$. So in particular, $G_d = \{u \in G~|~N(u) \subseteq [A]\}$.

\begin{definition}A set $F \subseteq G$ is a $\varphi$-approximation for $A$ if
\begin{enumerate}
\item $F \supseteq G_{\varphi}$
\item $N(F) \supseteq [A]$
\end{enumerate}
\end{definition}

Second, is notion of ``$\psi$-approximation''. Let $\psi = d/\log d$. 

\begin{definition}
For a $d$-regular bipartite graph with bipartition $(X,Y)$, we say that $(S,F) \in 2^X \times 2^Y$ is a $\psi$-approximation for $A$ if $S \supseteq [A]$, $F \subseteq G$ and the following two conditions hold:

\begin{enumerate}
\item $d_F(u) \geq d - \psi$ for every $u \in S$
\item  $d_{X \setminus S}(v) \geq d - \psi$ for every $v \in Y \setminus F$.
\end{enumerate}
\end{definition}

The following is a useful property of the $\psi$-approximation (Lemma $5.3$ in~\cite{GALVIN19}).

\begin{lemma}
\label{lem:property}
Let $(S,F)$ be a $\psi$-approximation for $A$. Then $|S| \leq |F| + 2\frac{t \psi}{d - \psi}$.
\end{lemma}

The second component to the proof of Lemma~\ref{lem:main} is the following:

\begin{lemma}
\label{lem:varphi} 
For every $C \in \mathcal{C}_1$, there is a family $\mathcal{C}_2(C) \subset 2^Y$ of size at most 
\[
2^{O\left(\frac{t }{\log d}\right)}
\] 
such that $\mathcal{C}_2(C)$ contains a $\varphi$-approximation for every small set $A\subset X$  whose boundary is contained in $C$. 
\end{lemma}

The requirement that $\Gamma$ is an ``expander'' is a crucial point in ~\cite{SAPOZHENKO87}, which does not apply to general Cayley graphs. As mentioned before, the main idea here is to try and overcome this by using $\mathcal{C}_1$. This lemma is the only place where we use the fact that $|D| \leq n^{1/3}$ and the fact that $A$ is small. Let us denote $\mathcal{C}_2 := \bigcup_{C \in \mathcal{C}_1}\mathcal{C}_2(C)$.

The final component to the proof of Lemma~\ref{lem:main} is the following:

\begin{lemma}
\label{lem:psi}
For every $F \in \mathcal{C}_2$, there is a family $\mathcal{C}_3(F) \subseteq 2^{X} \times 2^{Y}$ of size at most 
\[
2^{O\left(\frac{t\log^4 d}{\sqrt{d}}\right)}
\]
which contains a $\psi$-approximation for every $A$ such that $F \in \mathcal{C}_2$ which is a $\varphi$-approximation for $A$.
\end{lemma}

Define $\mathcal{C}_3 := \bigcup_{F \in \mathcal{C}_2}\mathcal{C}_3(F)$.

Given the above three lemmas, we prove Lemma~\ref{lem:main} as follows:

\subsection{Reconstruction: Proof of Lemma~\ref{lem:main} given Lemmas~\ref{lem:bdry},~\ref{lem:varphi}, and~\ref{lem:psi}}
\label{subsec:recon}

First,we upper bound the size of $\mathcal{C}_3$. Lemmas~\ref{lem:bdry},~\ref{lem:varphi}, and~\ref{lem:psi} imply that

\begin{align*}
|\mathcal{C}_3| & \leq 2^{O\left(\frac{t \log^4 d}{\sqrt{d}}\right)} \cdot |\mathcal{C}_2| \\
& \leq 2^{O\left(\frac{t \log^4 d}{\sqrt{d}}\right)} \cdot 2^{O\left(\frac{t}{\log d}\right)} \cdot |\mathcal{C}_1|\\
& \leq 2^{O\left(\frac{t \log^4 d}{\sqrt{d}}\right)} \cdot 2^{O\left(\frac{t}{\log d}\right)} \cdot 2^{O\left(\frac{t}{\log d}\right)} \\
& \leq 2^{O\left(\frac{t}{\log d}\right)}.
\end{align*}

Then, we use the following, which follows from methods of Kahn and Park~\cite{KP19}, stated explicitly by Park~\cite{PARK21}

\begin{lemma}
For each $\psi$-approximation $(S,F)$, there are at most 
\[
2^{g - \Omega(t)}
\]
sets $A$ such that $(S,F)$ is a $\psi$-approximation for $A$.
\end{lemma}

Thus we have

\begin{align*}
\mathcal{G}(a,g) & \leq 2^{g - \Omega(t)} \cdot |\mathcal{C}_3| \\
& \leq 2^{g - \Omega(t)}.
\end{align*}

Before we proceed to prove the lemmas~\ref{lem:bdry},~\ref{lem:varphi}, and~\ref{lem:psi}, we prove a few more preliminary results.

\subsection{More preliminaries}
\label{subsec:preliminaries}
We will use a consequence of Theorem~\ref{thm:PRGEN} which tells us that given a bipartite Cayley graph with bipartition $(X,Y)$, we can always select an almost spanning subset of a given vertex set in $X$ or $Y$ whose second and third neighborhood in comparison to the first neighborhood is not much larger. Note that the following is an easy observation, while Lemma~\ref{lem:PRUSE2} shows that in the trivial bound $|M+iD|\leq m+td^i$ can be improved if one chooses an appropriate large subset of $M$.

\begin{fact}
\label{thm:dvs2d}
If $M \subseteq X$ with $|M| = m$ and $|M+D| = m+t$, then for every $i \geq 2$, we have $|M + iD| \leq m + d^i \cdot t$.
\end{fact}

\begin{lemma}
\label{lem:PRUSE2}
Let $M \subseteq X$ with $|M| = m$, and $|M+D| = m + t$. Then, for any $k \in \mathbb{N}$ and $c \geq 4$, there is an $M^{(k)} \subseteq M$ with $|M \setminus M^{(k)}| \leq k \cdot \frac{t}{c}$ and 
\[
|M^{(k)} + (i+1)D| \leq m + (2i)^{i+1} \cdot c^i \cdot t
\]
for each $i \leq k$.
\end{lemma}

\begin{proof}
We prove this by induction on $k$. 

Set  $N = M + \{e\}$ for some $e \in D$, $j=1$, $h = 2$, $s=t$, and $\ell = m - \frac{t}{c}$. Since $M+D \supseteq N$, and $|N| = |M|$, we have that $|(M+D)\setminus N| = t$. Theorem~\ref{thm:PRGEN} guarantees the existence of a set $M^{(1)} \subseteq M$ such that $|M \setminus M'| \leq \frac{t}{c}$ and 

\begin{align*}
|(M^{(1)} + 2D) \setminus (N+D)| & \leq \frac{t^2}{2}\left(\frac{c}{t} - \frac{1}{m}\right) + c^2\left(m - \left(m - \frac{t}{c}\right)\right)\\
& \leq 2c \cdot t.
\end{align*}
Therefore, $|M^{(1)} + 2D| \leq |(M^{(1)} + 2D) \setminus (N+D)|  + |N+D| \leq m + t+ 2c t \leq m+ 3ct$ which completes the base case of the induction.

For $i \geq 2$, let us assume that there is an $M^{(i)} \subseteq M^{(i - 1)} \cdots M^{(1)} \subset M$ such that 
\begin{itemize}
\item $|M^{(i'-1)} \setminus M^{(i')}| \leq \frac{t}{c}$ 
\item $|M^{(i')} + (i'+1)D|\leq  |M^{(i')}| + (2i')^{i'+1} \cdot c^{i'} t $
\end{itemize}

For each $i' \leq i$. Set $N = M^{(i)} + \{e\}$ for some $e \in D$. Since $M^{(i)} + (i+1)D \supseteq N+ iD$, we have 

\begin{align*}
|(M^{(i)} + (i+1)D) \setminus (N + iD)| & = |M^{(i)} + (i+1)D| - |N+ iD| \\
& = |M^{(i)} + (i+1)D| - |M^{(i)} + iD| \\
& \leq |M^{(i)} + (i+1)D| - |M^{(i)}| \\
& \leq (2i)^{i+1} \cdot c^{i} t .
\end{align*}

Now apply Theorem~\ref{thm:PRGEN} with $M=M^{(i)}$, $N = M^{(i)} + \{e\}$ for some $e \in D$, $j = i+1$, $h = i+2$, and $\ell = |M^{(i)}| - \frac{t}{c}$, $s:=|(M^{i} + (i+1)D) \setminus (N + iD)|\leq  (2i)^{i+1} \cdot c^i t$.
So again, we obtain a set $M^{(i+1)} \subseteq M^{(i)}$ of size at least $\ell \geq |M^{(i)}| - \frac{t}{c}$ and

\begin{align*}
&|(M^{(i+1)} + (i+2)D) \setminus (N + (i+1)D)| \\
& \leq \frac{((2i)^{i+1} \cdot c^i t)^{\frac{i+2}{i+1}}}{(i+2)/(i+1)}\left(\left(\frac{c}{t}\right)^{\frac{1}{i+1}} - \left(\frac{1}{m}\right)^{\frac{1}{i+1}}\right) + (2i)^{i+2}\cdot c^{i+2}\left(|M^{(i)}| - \left(|M^{(i)}| - \frac{t}{c}\right)\right)\\
& \leq  2 \cdot (2i)^{i+2} \cdot c^{i+1} t,  \\
\end{align*}

and therefore,
\begin{align*}
|(M^{(i+1)} + (i+2)D)| & \leq |N + (i+1)D| + 2 \cdot (2i)^{i+2} \cdot c^{i+1} t \\
& = |M^{(i)} + (i+1)D| +  2 \cdot (2i)^{i+2} \cdot c^{i+1} t  \\
& \leq |M^{(i)}| + (2i)^{i+1}c^{i} t + 2 \cdot (2i)^{i+2} \cdot c^{i+1} t \numberthis \label{eqn:ind2}\\
& \leq |M^{(i+1)}| + \frac{t}{c} +  (2i)^{i+1}c^{i} t + 2 \cdot (2i)^{i+2} \cdot c^{i+1} t \numberthis \label{eqn:ind1}\\
& \leq  |M^{(i+1)}| + (2(i+1))^{i+2}c^{i+1}t.
\end{align*}

Here, (\ref{eqn:ind2}) follows from the induction hypothesis, and (\ref{eqn:ind1}) follows from $|M^{i+1} \setminus M^{(i)}|\leq \frac{t}{c}$. Thus we have  $M^{(k)} \subseteq \cdots \subseteq M^{(1)} \subseteq M$ and for every $i \leq k$, we have $|M^{(i-1)} \setminus M^{(i)}| \leq \frac{t}{c}$ and 
\begin{align*}
|M^{(k)} + (i+1)D| & \leq |M^{(i)}+(i+1)D|  \\
&  \leq |M^{(i)}| + (2i)^{i+1} \cdot c^i t \\
& \leq m + (2i)^{i+1}\cdot c^i t
\end{align*} 
for each $i \leq k$ as claimed.
\end{proof}

Next, as a corollary of Theorem~\ref{thm:PR} and Theorem~\ref{thm:CDGEN}, we have the following:

\begin{corollary}
\label{corr:basicexpansion}
Let $M \subseteq X$ such that $|M| \leq |X|/2$ and $|M+D| = \alpha|M|$, and $M + 2D \neq X$. Then $|2D| \leq 2(\alpha^2 - 1)|M|$.
\end{corollary}

\begin{proof}
A direct application of Theorem~\ref{thm:PR} gives us an $M' \subseteq M$ that 
\begin{equation}
\label{eqn:PR}
|M'+2D| \leq \alpha^2|M'|.
\end{equation}
 Since $\Gamma$ is connected, $D$ is a generating set, and so we must have that $|M' + 4D| > |M' + 2D|$. Indeed, suppose otherwise, then it must be the case that $|M' + 2D| = |M' + 3D| = |M' + 4D|$, since for any two sets $A,B \in \mathcal{F}$, $|A+2B| \geq |A + B|$. But since $M' + 3D = N_{\Gamma}(M' + 2D)$, this gives us that $M' + 2D = X$ or $\Gamma$ is disconnected, which is a contradiction.  So Theorem~\ref{thm:CDGEN} gives us that 
 \begin{equation}
 \label{eqn:olsen}|M' + 2D| \geq |M'| + (1/2) \cdot |2D|.
 \end{equation}
 
 Combining~(\ref{eqn:PR}) and~(\ref{eqn:olsen}) gives us that $|2D| \leq 2(\alpha^2 - 1)|M'| \leq 2(\alpha^2 - 1)|M|$.
\end{proof}

We also need the following easy observation.

\begin{proposition}
\label{claim:triv}
Let $D'\subseteq D$ such that $|D \setminus D'|\leq \sqrt{d}/ \log d$. Then $|D+D'| \geq (|2D|)(1 - \left(1/\log^2  d\right))$.
\end{proposition} 

\begin{proof}
We have that $2D = (D + D') \cup (2 \cdot(D \setminus D'))$. Since $|2 \cdot(D \setminus D')| \leq |D \setminus D'|^2 \leq |D|/\log^2 |D|$, the claim follows.
\end{proof}

Corollary~\ref{corr:basicexpansion} also gives us the following, which is the only place we use the fact that $|D| \leq n^{1/3}$.

\begin{corollary}
\label{corr:2nbhd}
Let $D' \subset D$ such that $|D'| \geq d- \sqrt{d}/\log d$, and let $M \supseteq \{u\} + D'$ and $|M| \leq |X|/2$ for some $u \in \mathcal{F}$. Then the following holds: 
\[|M+D| \geq |M| + |2D|/6.\]
\end{corollary}

\begin{proof}
The statement is clearly true for $|M| \leq (1/6) \cdot |2D|$ because of Proposition~\ref{claim:triv}, since $|M+D| \geq |D'+D| \geq |2D|(1 - 1/(\log^2|D|)) \geq |M| + |2D|/6$.

For $|M| > (1/6)\cdot |2D|$, suppose we had $|M + D| < |M| + |2D|/6$. Then Fact~\ref{thm:dvs2d} gives us that 

\begin{align*}
|M + 2D| & < |M| + |D| \cdot |2D|/6 \\
& \leq |M| + |D|^3/6 \\
& \leq |X|/2 + |X|/6 \\
&< |X|.
\end{align*}

This is the only place we use $|D| \leq n^{1/3}$. Since $M$ satisfies the hypothesis of Corollary~\ref{corr:basicexpansion}, we have that \

\begin{align*}
|2D| & \leq 2\left(\left(1 + \frac{|2D|}{6|M|}\right)^2 - 1\right)|M| \\
& < 2\left(3 \cdot \frac{|2D|}{6|M|}\right)|M| \\
& =  |2D|
\end{align*}
which is a contradiction. The second inequality is because $(1 + x)^2 < 1 + 3x$ for $x \in (0,1)$.
\end{proof}

\subsection{Boundaries: Proof of Lemma~\ref{lem:bdry}}
\label{subsec:bdry}

Applying Lemma~\ref{lem:PRUSE2} setting $M = [A]$ and $c = \log^2 d$, there is an $\underline{A} \subseteq [A]$ such that 
\begin{itemize}
\item[1.] $|[A] \setminus \underline{A}| \leq 4\frac{t}{\log^2 d}$
\item[2.] $|N^i(\underline{A})| \leq a + O(t\log^{2(i-1)} d)$ for each $i \in [4]$.
\end{itemize}

Moreover, we may assume that $\underline{A} = [\underline{A}]$. Suppose not, replacing $\underline{A}$ by $[\underline{A}]$ (which is possible because $[A]$ is closed and so, $[\underline{A}] \subseteq [A]$) does not violate either of the properties. Define $\underline{G} := N(\underline{A})$ and $\underline{G}' := \{u \in N(\underline{A})~|~N(u) \cap \underline{A}^c \neq \emptyset\}$ to be the boundary of $\underline{G}$. Observe that 

\begin{equation}
\label{eqn:split}
G' \subseteq \underline{G}' \cup N([A] \setminus \underline{A}).
\end{equation}

Before we proceed, let us make a few definitions. Define $G_0 := N^3(\underline{A}) \setminus \underline{G}$, $A_0 := N^2(\underline{A}) \setminus \underline{A}$, and $A_1 := N^4(\underline{A}) \setminus N^2(\underline{A})$. Lemma~\ref{lem:PRUSE2} implies that $|A_0| = O(t \log^2 d)$, $|G_0| = O(\log^4 d)$, and $|A_1| = O(\log^{6} d)$.

We have $N(G_0) \subseteq A_0 \cup A_1$. So, by Theorem~\ref{thm:cover}, there is a set $Z_1 \subset A_0 \cup A_1$ such that

\[
|Z_1| \leq O\left(\frac{|A_0 \cup A_1|\log d}{d}\right) = O\left(\frac{t \log^{7} d}{d}\right)
\]

and 

\begin{equation}
\label{eqn:Z}
G_0 \subseteq N(Z_1).
\end{equation}

Let $\underline{G}' = G_{L} \sqcup G_{S}$, where 
\[
G_{S} := \{v \in \underline{G}'~|~ d_{A_0}(v) \geq  d/2\}.
\]

Since we have $|A_0| = O(t \log^2 d)$, by Theorem~\ref{thm:cover}, there is a subset $Z_2 \subset A_0$ of size at most $O\left(\frac{t\log^3 d}{d}\right)$ such that

\begin{equation}
\label{eqn:G_S_b}
G_{S} \subseteq N(Z_2).
\end{equation}

Applying Lemma~\ref{lem:PRUSE2} with $M = \underline{G}^c$ gives that there is a subset $M' \subseteq \underline{G}^c$ such that $|M'| \geq |\underline{G}^c| - \frac{2t}{\log^3 d}$ and $|M' + 3D| \leq |\underline{G}^c| + O(t\log^6 d)$. Set $A_2 := |M'+3D| \cap \underline{A}$. We have that $|A_2| = O(t \log^6 d)$.

Let $G'' := G_L \cap N^2(M')$. We have that each vertex in $G''$ must have at least $d/2$ neighbors in $A_2$. So, by Theorem~\ref{thm:cover}, there is a subset $Z_3 \subset A_2$ of size at most $\frac{|A_2| \log d}{d} = O\left(\frac{t \log^7 d}{d}\right)$ such that

\begin{equation}
\label{eqn:G''_b}
G'' \subseteq N(Z_3).
\end{equation}

Since $\underline{A}$ is closed, we have $\underline{A}^c = N(\underline{G}^c)$. So, $G_L \subseteq N^2\left(\underline{G}^c\right)$ and therefore, 
\begin{equation}
\label{eqn:G_L-G''_1}
G_L \setminus G'' \subseteq N^2(\underline{G}^c \setminus M'). 
\end{equation}

Moreover, every $u \in G_L \setminus G''$ satisfies $(N^2(u) \cap \underline{G}^c) \subseteq N(A_0) = G_0$, and therefore,

\begin{equation}
\label{eqn:G_L-G''_2}
N^2(G_L\setminus G'') \cap \underline{G}^c \subseteq G_0.
\end{equation}

Taking (\ref{eqn:G_L-G''_1}) and (\ref{eqn:G_L-G''_2}) together, we have that

\begin{equation}
\label{eqn:G'_L-G''_b}
G_{L} \setminus G'' \subseteq N^2\left(\left(\underline{G}^c \setminus M'\right) \cap G_0\right) .
\end{equation}

Thus, since we have

\[
\underline{G}' = G_S \sqcup G'' \sqcup (G_L \setminus G''),
\]

we have, using (\ref{eqn:G_S_b}),  (\ref{eqn:G''_b}), and (\ref{eqn:G'_L-G''_b}),

\begin{equation}
\label{eqn:underlineG'}
\underline{G}' \subseteq N(Z_2) \cup  N(Z_3) \cup N^2\left(\left(\underline{G}^c \setminus M' \right)\cap G_0\right). 
\end{equation}

For $([A] \setminus \underline{A})$, we have the following:

\begin{claim}
\label{claim:underline}
The number of possibilities for $[A]\setminus \underline{A}$'s for a given $Z_1$ is at most $2^{O\left(\frac{t }{\log d}\right)}$.
\end{claim}

\begin{proof}
Since $[A]$ is $2$-linked, we must have that every $2$-linked component in $[A] \setminus \underline{A}$ has at least one vertex in with $N^2(\underline{A}) \setminus \underline{A}$. Since we have that
\[
N^2(\underline{A}) \setminus \underline{A} = A_0 \subseteq N(G_0) \subseteq N^2(Z_1),
\] 
we can choose $[A]\setminus \underline{A}$ from a given $Z_1$ by the following procedure: (1) Choose one vertex per $2$-linked component of $[A] \setminus \underline{A}$  from $N^2(Z_1)$, (2) specify the sizes of these $2$-linked components and finally, (3) specify the vertices in each of these components by specifying the BFS tree starting from the chosen vertices in some predetermined order. 

The first can be done in $\binom{|N^2(Z_1)|}{\leq \frac{4t}{\log^2 d}}$ ways, the second in $2^{\frac{8t}{\log^2 d}}$ ways, and the third, using Proposition~\ref{prop:trees}, in $d^{\frac{8t}{\log^2 d}}$ ways.

Since $|Z_1| = O\left(\frac{t \log^{7} d}{d}\right)$, we have that $|N^2(Z_1)| \leq d^2|Z_1| = O(td \log^{7} d)$. Therefore, the total number of choices for $[A]\setminus \underline{A}$ is at most $2^{O\left(\frac{t}{\log d}\right)}$.
\end{proof} 

Recalling (\ref{eqn:split}) and (\ref{eqn:underlineG'}), we have that

\[
G' \subseteq N(Z_2) \cup  N(Z_3) \cup N^2\left(\left(\underline{G}^c \setminus M' \right)\cap G_0\right) \cup N([A] \setminus \underline{A})
\]

The size of each possible $([A] \setminus \underline{A})$ described by Claim~\ref{claim:underline} is at most $\frac{4t}{\log^2 d}$. Each of the sets $Z_2$, and $Z_3$ are of size at most $O\left(\frac{t \log^7 d}{d}\right)$. Finally, $(\underline{G}^c \setminus M') \cap G_0$ is a set of size at most $|\underline{G}^c \setminus M'| \leq \frac{t}{\log^4 d}$. Putting these together, we have
\begin{align*}
& |N(Z_2) \cup N^2((\underline{G}^c \setminus M') \cap G_0) \cup N(Z_3) \cup N([A] \setminus \underline{A})|  \\
& \leq  d|Z_2| + d|Z_3| + d_2|\underline{G}^c \setminus M'| + d|[A] \setminus \underline{A}|\\
& = O\left(\frac{td_2}{\log^3 d}\right).
\end{align*}

To count the number of possibilities for this, each tuple $\left(Z_2,Z_3,(\underline{G}^c \setminus M'),([A] \setminus \underline{A})\right)$ is described as follows:

\begin{itemize}
\item The sets $Z_2$, and $Z_3$ are specified explicitly by sets of size $O\left(\frac{t \log^7 d}{d}\right)$ each. This gives at most \[
\binom{n}{O\left(\frac{t \log^7 d}{d}\right)}^2
\]
possibile descriptions. 
\item The set $(\underline{G}^c \setminus M') \cap G_0$ is specified by 
\begin{itemize}
\item Specifying $Z_1$, which is a set of size $O\left(\frac{t \log^{7} d}{d}\right)$. This has 
\[
\binom{n}{O\left(\frac{t \log^{7} d}{d}\right)}
\]
possible descriptions.
\item Specifying the subset of $N(Z_1)$ of the size at most $|\underline{G}^c \setminus M'| \leq \frac{2t}{\log^2 d}$. This has at most $\binom{|N(Z_1)|}{2t/\log^2 d} = 2^{O\left(\frac{t}{\log d}\right)}$ possible descriptions.
\end{itemize}
\item Specifying $[A] \setminus \underline{A}$ as in Claim~\ref{claim:underline} using $Z_1$, which has at most $2^{O\left(\frac{t}{\log d}\right)}$ descriptions.
\end{itemize}

So in total, the number of possible descriptions (and therefore, the number) of tuples $\left(Z_2,Z_3,(\underline{G}^c \setminus M'),([A] \setminus \underline{A})\right)$ is at most

\[
\binom{n}{\frac{t \log^7 d}{d}}^2 \cdot \binom{n}{\frac{t \log^{7} d}{d}} \cdot 2^{O\left(\frac{t}{\log d}\right)} = 2^{O\left(\frac{t}{\log d}\right)}
\]

Where in the last equality, we have used $d/ \log^{8} d = \Omega(\log n)$.

\subsection{$\varphi$-approximation: Proof of Lemma~\ref{lem:varphi}}
\label{subsec:varphi}

We will first pre-process the graph using the following contraction algorithm given a $C \in \mathcal{C}_1$. 

\begin{itemize}
\item $R \gets X$, $B\gets \emptyset$
\item While $Y \supsetneq C$ do
\begin{itemize}
\item Let $u  \in Y \setminus C$ be arbitrary.
\item $Y \gets Y \setminus \{u\}$.
\item $X \gets X \setminus N(u) \cup \{v'\}$ where $N(v') = \cup_{v \in N(u)}N(v)$ with multiplicities.
\item $R \gets R\setminus N(u)$, $B \gets B \setminus N(u) \cup \{v'\}$
\end{itemize}
\end{itemize}

Let $\mathcal{F}'$ be the final graph after the algorithm terminates with parts $X'$ and $Y'$. The set $R$ consists of all vertices whose neighbors are all in $C$, and the set $B$ are all vertices obtained through the above mentioned contraction algorithm. Since at each step, the algorithm contracts a vertex in $Y$ with it's (current) neighborhood, every vertex in $B$ is obtained via the contraction of \emph{all} the vertices in $N(S)_d$ for some $S \subseteq X$. Thus, the set $B$ is given by $\{v_S\}$ where $S \subset X$ and $v_S$ corresponds to the subset $S \cup N(S)_d$, and 
\begin{equation}
\label{eqn:Nv_S}
N(v_S) = N(S)\setminus N(S)_d. 
\end{equation}

Before the start of the algorithm, every vertex in $Y\setminus C$ has all its neighbors either in $[A]$ or $[A]^c$. Consider the partition $B = B_{A} \sqcup B_{A^c}$ defined as $B_A := \{v_S \in B~|~S \subset [A]\}$. Similarly, partition $R = R_A \sqcup R_{A^c}$ where $R_A = R \cap [A]$. 

We have that for every set $S$, $|N(S)_0| \leq |S|$, and $S \supseteq N(u)$ for every $u \in N(S)_d$. Moreover, $A$ is small. Thus Corollary~\ref{corr:2nbhd}, and (\ref{eqn:Nv_S}) together imply that every vertex in $B_A$ has degree at least $d_2/6$.

Define $Q_0$ to be a $p = \left(\frac{60\log d}{d_2}\right)$-random subset of $Y \cap G$. The following four properties hold with probability at least $1/5$.

\begin{enumerate}
\item $|Q_0| \leq O\left(\frac{t}{\log^2 d}\right)$.
\item $\nabla(Q_0, (R_{A^c} \cup B_{A^c}))  = O\left(\frac{t}{\log^2 d}\right)$.
\item $\#\{u \in B_A~|~Q_0 \cap N(u) = \emptyset \} = O\left(\frac{t}{d^7}\right)$.
\item $|(G_{\varphi}\cap C)  \setminus N(N_{R_A \cup B_A}(Q_0))| = O\left(\frac{t}{d^8}\right)$.
\end{enumerate}

First we observe that $\mathbf{E}[|Q_0|] \leq p|C|$. Thus the probability that Property $1.$ does not hold is at most, using Markov's inequality, $1/5$. Next, we observe  
\[
\mathbf{E}[|\nabla(Q_0, (R_{A^c} \cup B_{A^c}))|] = ptd = \frac{td \log d}{d_2} \leq \frac{10t}{\log^2 d}.
\] 
Thus the probability that Property $2.$ does not hold is, again by Markov's inequality, at most $1/5$.

Define $Q_1 :=  \nabla(Q_0, (R_{A^c} \cup B_{A^c}))$ . 

For property $3.$, we use the fact that every vertex in $B_A$ has degree at least $\Omega(d_2)$. So for each $u \in B_A$, we have $\mathbb{P}(Q_0 \cap B_A = \emptyset) \leq (1-p)^{d_2/6} \leq d^{-10}$. Moreover, after the algorithm $|X'| \leq d|C| \leq td^{3}$. So we have
\[
\mathbf{E}[\#\{u \in B_A~|~ Q_0 \cap N(u) = \emptyset \}] \leq td^{-7}.
\]
Therefore, the probability that Property $3.$ does not hold is again at most $1/5$.

Define $Q_2 := \{u \in B_A~|~ Q \cap N(u) = \emptyset \}\}$.

Let us abbreviate $C' := C \cap G_{\varphi}$. Property $4.$ follows from using the fact that for every $u \in C'$, $| N(N_{R_A \cup B_A}(u))| \geq d_2/6$. Indeed, since any $u \in C'$ has at least $\varphi$ edges to $A$, and therefore to $R_A \cup B_A$. Thus we may apply Corollary~\ref{corr:2nbhd} to obtain the desired bound on $| N(N_{R_A \cup B_A}(u))|$. So for any given $u$, we have that $\mathbb{P}(Q_0 \cap  N(N_{R_A \cup B_A}(u)) = \emptyset) \leq (1 - p)^{d_2/6} \leq d^{-10}$, and so

\[
\mathbf{E}[\#\{u ~|~ N(N_{R_A \cup B_A}(u)) \cap Q_0 = \emptyset\}] \leq |C| \cdot d^{-10} \leq td^{-8}.
\]
So the probability that Property $4.$ does not hold is again at most $1/5$.

Define $Q_3 := C'\setminus  N(N_{R_A \cup B_A}(Q_0))$.

Finally, by the Union Bound, the probability that either of the properties does not hold is at most $4/5$, and so in particular, there is a choice for $Q_0$ (and therefore, for $Q_1$, $Q_2$, and $Q_3$) that satisfies all four properties.

We claim that given $Q_0$, $Q_1$, $Q_2$, and $Q_3$, one can construct a set $Z_1 \subseteq G$ such that $Z_1 \supseteq G_{\varphi}$. Indeed, since using $Q_0$ and $Q_1$, one can construct $N(N_{R_A \cup B_A}(Q_0)) \subseteq G$. So far, this is missing all the vertices in $Q_3$ and $G_d \setminus C$. We have $Q_3$ provided, and finally, by $Q_2$, and the neighbors of $Q_0$, one can determine $B_A$, and therefore, $G_d \setminus C$.

Now the only vertices in $A$ uncovered by $Z_1$ are $R_A \setminus N(Z_1)$, since by construction, $N(Z_1)\supset B_A$. Note that every vertex in $R_A \setminus N(Z_1)$ has degree $d$ to $G \setminus Z_1$. Moreover, $|G \setminus Z_1| \leq t\sqrt{d}\log d$. This is because $Z_1 \supseteq G_{\varphi}$, and each vertex in $G \setminus Z_1$ contributes at least $\sqrt{d}/ \log d$ edges to $\nabla(G,[A]^c)$, which is a set of size $t d$. Thus by Theorem~\ref{thm:cover}, can specify a further $O\left(\frac{t \log^3d}{\sqrt{d}}\right)$ vertices in $C$ such that $N(C) \supset R_A \setminus Z_1$. Let $Z_2$ denote this set of vertices. The final $\varphi$-approximation is $Z_1 \cup Z_2$. We count the number of these as follows:

\begin{enumerate}
\item The set $Q_0$, $Q_3$, and $Z_2$ are subsets of $C$, each of size at most $\frac{50t}{\log^2 d}$, so the number of choices for these sets are at most $\binom{|C|}{\frac{50t }{\log^2 d}}^3 \leq 2^{O\left(\frac{t}{\log d}\right)}$.
\item Since $|Q_2| \leq \frac{5t}{d^7}$, the number of choices for this is at most $\binom{|X'|}{\frac{5t}{d^7}} \leq \binom{td^3}{\frac{t}{d^7}}= 2^{O\left(\frac{t \log d}{d^{7}}\right)}$. 
\item Finally, the number of choices for $Q_1$ are at most $\binom{td}{|Q_1|}\leq 2^{O\left(\frac{t}{\log d}\right)}$.
\end{enumerate}

Thus for every $C \in \mathcal{C}_1$, there is a set of at most $2^{O\left(\frac{t}{\log d}\right)}$ many $\varphi$-approximations for all sets $A$ such that $C$ contains the boundary of $A$.

\subsection{$\psi$-approximation: Proof of Lemma~\ref{lem:psi}} 
\label{subsec:psi}
Recall that $\psi = \frac{d}{\log d}$ and $\varphi = d - \frac{\sqrt{d}}{\log d}$. Fix an order $\ll$ on $X \cup Y$ and do the following procedure
\begin{itemize}
\item Initialize $F' \gets F$.
\item While $\mathcal{Q} := \{u \in [A]~|~d_{G \setminus F'}(u) \geq \psi\} \neq \emptyset$ do
\begin{itemize} 
\item Let $u \in \mathcal{Q}$ be smallest w.r.t. $\ll$
\item $F' \gets F' \cup N(u)$.
\end{itemize}
\item Initialize $F'' \gets F'$, $S \gets \{u \in X~|~d_{F''}(u) \geq d - \psi\}$
\item While $\mathcal{Q}' := \{w \in Y~|~d_{S''}(w) > \psi\} \neq \emptyset$ do
\begin{itemize}
\item $w \in \mathcal{Q}'$ be smallest w.r.t. $\ll$
\item $S'' \gets S'' \setminus N(w)$.
\end{itemize}
\item return $(S'',F'')$.
\end{itemize}

The fact that $(S'',F'')$ is a $\psi$-approximation can be verified easily. We will only focus on enumerating the number of such pairs for a given $F$. Every pair is determined completely by the set of vertices $u$ chosen in the first loop and the set of vertices $w$ chosen in the second.

First, we observe that before the first loop, $td = \nabla(G, [A]^c) \geq |G \setminus F'| \cdot (d - \varphi)$, and so $|G \setminus F'| \leq \frac{td }{d - \varphi}$. So in the first loop, each $u \in \mathcal{Q}$ removes at least $\psi$ vertices from this set, and therefore, the first loop is run for at most $\frac{t d}{\psi \cdot (d - \varphi)}$ times. Moreover, we have 
\[
\mathcal{Q} \subseteq N(G \setminus F') \subseteq N(N^2(F') \setminus F') \subseteq N(N^2(G) \setminus F')
\] 

where the second containment follows since $F' \supseteq F$. So by Theorem~\ref{thm:dvs2d}, using the fact that $N^2(G) = N^3(A)$, we have 
\[
|\mathcal{Q}| \leq d(|N^2(G)| - |F'|) \leq O(td^4).
\] 
Therefore, the number of ways of choosing the vertices $u$ from $N(N^2(F')\setminus F')$ in the first loop is at most 
\[
\binom{O(td^4)}{\leq \frac{td}{\psi \cdot (d -  \varphi)}} \leq 2^{O\left(\frac{t d \log d}{\psi \cdot (d - \varphi)}\right)}.
\]

Next, we observe that before the second loop, $td = \nabla(G,[A]^c) \geq |S'' \setminus A| (d - \psi)$ and so $|S''\setminus A| \leq td/(d - \psi)$. So in the second loop, each $w \in \mathcal{Q}'$ removes at least $\psi$ vertices from this set, and therefore, the second loop is run for at most $\frac{td}{\psi(d - \psi)}$ times. Moreover, we have
\[
\mathcal{Q}' \subseteq N^2(G \setminus F') \subseteq N^2(N^2(F') \setminus F') \subseteq N^2(N^2(G) \setminus F')
\]
Where the second containment follows since $F'' \supseteq F$. So by Theorem~\ref{thm:dvs2d}, using the fact that $N^2(G) = N^3(A)$, we have 
\[
|\mathcal{Q}'| \leq d^2(|N^2(G)| - |F'|) \leq O(td^5).
\] 
Therefore, the number of ways of choosing the vertices $w$ from $N^2(N^2(F')\setminus F')$ in the first loop is at most 
\[
\binom{O(td^5)}{\leq \frac{td}{\psi (d - \psi)}} \leq 2^{O\left(\frac{t d \log d}{\psi(d - \psi)}\right)}.
\]

which completes the proof.

\section{Acknowledgements}
We would like to thank Will Perkins for various insightful conversations on the topic. We are also grateful to Dhruv Mubayi, who asked the question on the number of independent sets in Cayley graphs on $\mathbb{Z}_{2n}$ (see Appendix B), which essentially motivated our research on this problem.

\bibliography{references}
\bibliographystyle{alpha}

\appendix

\section{Connected graphs with many independent sets}
\label{sec:conn}

In this section, we sketch the construction of a $d$-regular graph on $2n$ vertices which is $(d-1)$-connected and has $2^{n + \Omega(n/d)}$ independent sets.

Let $n$ and $d$ be such that $(4d-2)|n$ and let $t := \frac{n}{4d - 2}$. For $i \in \left[t\right]$ let $H_i$ be a bipartite graph with bipartition $(X_i,Y_i \sqcup Z_i)$ where each $|X_i| = 2d - 2$, and each $|Y_i| = |Z_i| = d$ such that the following properties hold for each $i$:

\begin{itemize}
\item[1.] $d(u) = d$ for $u \in X_i$.
\item[2.] $d(v) = d - 1$ for $v \in Y_i \cup Z_i$.
\item[3.] $H_i$ is $(d-1)$-connected.
\end{itemize}

Let $G$ be a graph obtained by placing a matching between $Z_i$ and $Y_{(i+1) \mod t}$ for each $i \in \left[t\right]$. \\

We have that $G$ is a $d$-regular $(d-1)$-connected bipartite graph on $2n$ vertices. 

Let $L_i := X_i$ if $i$ is odd and $Y_i \cup Z_i$ if $i$ is even. Similarly, let $R_i := Y_i \cup Z_i$ if $i$ is odd and $X_i$ if $i$ is even.

We say \emph{interval} of $[t]$ to mean a subset of consecutive integers. Let $S\subseteq [t]$ be a collection of $c$ distinct intervals where each interval starts and ends on a distinct odd number. Suppose that $c = \delta t$ for some fixed (TBD) constant $\delta$. Define

\[
M(S) := \bigcup_{i \in S}L_i \cup \bigcup_{i \not \in S}R_i
\]

\begin{claim}For each such $S$, $M(S)$
 is a maximal independent set of size at least $n - 2c$. 
\end{claim}

\begin{proof}
Let $C \subset S$ be an interval that starts and ends on an odd number, and let $X = \bigcup_{i \in C}L_i$. From the construction of $G$, have $N(X) = \bigcup_{i \in C}R_i$, and so
\[
N\left(\bigcup_{i \in S}L_i\right) = \bigcup_{i \in S}R_i.
\]
Therefore, $\bigcup_{i \in S}L_i \cup \bigcup_{i \not \in S}R_i$ is a maximal independent set.

Moreover, from construction, we have $|N(X)| = |X| + 2$. Therefore, 
\[
\left|N\left(\bigcup_{i \in S}L_i\right)\right| = \left|\bigcup_{i \in S}L_i\right| + 2c,
\]
and so $\left|\bigcup_{i \not \in S}R_i\right| = n - \left|\bigcup_{i \in S}L_i\right| - 2c$, which completes the proof.
\end{proof}
 
 There are at least $\binom{t/2}{2c} \geq 2^{\Omega(\log(1/\delta))c}$ such $S$'s obtained by choosing the endpoints of the $c$ intervals. Let
 
 \[
 I(S) := \{I \in \mathcal{I}(G)~|~\forall i \in S,~L_i \cap I \neq \emptyset~\text{and }\forall i \not \in S,~R_i \cap I \neq \emptyset\}.
 \]
 
 For distinct $S_1$, $S_2$, we have that $I(S_1) \cap I(S_2) = \emptyset$, and so

\[
i(G) \geq \sum_{S}|I(S)|.
\]

 With this in mind, we have
 \begin{align*}
 |I(S)| & = \prod_{i \in S}\left(2^{|L_i|} - 1\right) \cdot \prod_{i \not \in S}\left(2^{|R_i|} - 1\right) \\
 & = \prod_{i \in S}2^{|L_i|}\left(1  - \frac{1}{2^{|L_i|}}\right) \cdot \prod_{i \not \in S}2^{|R_i|}\left(1 - \frac{1}{2^{|R_i|}}\right) \\
 & \geq \left(\prod_{i \in S}2^{|L_i|} \cdot \prod_{i \not \in S}2^{|R_i|}\right) \cdot  \left(1 - \frac{1}{2^{d}}\right)^{t} \\
 & \geq 2^{n - 2c - O\left(\frac{t}{2^d}\right)}
 \end{align*}
distinct independent sets for a small enough $\delta$, and so 

\begin{align*}
i(G) & \geq \sum_{S}|I(S)| \\
& \geq 2^{\Omega(\log (1/\delta) c)} \cdot 2^{n - 2c - O\left(\frac{t}{2^d}\right)} \\
& = 2^{n + \Omega(c)} \\
& = 2^{n + \Omega\left(\frac{n}{d}\right)}
\end{align*}
for a small enough $\delta$.

\section{Cayley graphs on $\mathbb{Z}_{2n}$}
\label{sec:appendix}

Here, we describe a Cayley graph on $2n$ vertices, degree $(2 - o(1)) \log n$ and $\omega(2^n)$ independent sets.

Fix an $\epsilon > 0$. Take $\Gamma$ to be a Cayley graph over $\mathbb{Z}_{2n}$ with the generator set $\{-d + 2i~|~ 0 \leq i \leq d\}$ for any odd integer $d \geq (1 + \epsilon)\log n$. Let $X$ and $Y$ be the sides of the bipartite graph, with $|X| = |Y| = n$. Observe that $X$ and $Y$ are the cosets of the subgroup of order $n$, i.e., $X$ and $Y$ partition $\mathbb{Z}_{2n}$ into `even' and `odd' elements respectively.

Observe that for each small $2$-linked set $A \subseteq X$, we have that $G = N(A)$ is just an arithmetic progression of common difference $2$. Thus, we can define the \emph{start} and \emph{end} of $G$ as the first and the $|G|$'th element respectively in this progression. Moreover, $|G| = |[A]| + d$. 

Enumerating the number of small $2$-linked $A$'s such that $|[A]| = a$ and $|G| = g$ where $g - a = d$ can be done as follows: Let $u,v \in [A]$ be the vertices that cover the \emph{start} and \emph{end} of $G$. Observe that every vertex in $G \setminus (N(u) \cup N(v))$ has at least $d$ neighbors in $[A]$. Thus a uniformly random subset of $[A]$ covers $G \setminus (N(u) \cup N(v))$ with probability at least $(1 - n\cdot 2^{-d}) = (1 - o(1))$. Thus at least $(1/4 - o(1))$ fraction of subsets of $[A]$ have $G$ as their neighborhood. Thus we have

 \begin{equation*}
  \mathcal{G}(a,g) = \begin{cases}
        \Theta\left(n \cdot 2^{g - d}\right) & \text{if $g-a = d$}\\
        0 &\text{otherwise}.
        \end{cases}
 \end{equation*}

One may verify that plugging this bound in the proof of Theorem~\ref{thm:main} gives that $i(\Gamma) \leq 2^{n+1}(1 + o(1))$ whenever $d\geq (2 + \epsilon)\log n$.

When $d \leq (2 -\epsilon) \log n$ we have:

\begin{align*}
i(\Gamma) & \geq \sum_{A \subseteq X,~\text{small}}2^{n - |N(A)|} \\
& = 2^{n}\sum_{A \subseteq X,~\text{small}}2^{-N(A)} \\
& \geq 2^{n}\left(1 + \sum_{\substack{\emptyset \neq A \subseteq X,~\text{small} \\2-\text{linked}}}2^{-N(A)} \right)\\
& \geq \Omega\left( 2^n \cdot \left(\sum_{g = d}^n n\cdot 2^{- d}\right)\right) \\
& = \omega\left(2^n\right).
\end{align*}

\end{document}